\documentclass{amsart}
\usepackage[all]{xy}
\usepackage{color}
\usepackage{amsthm}
\usepackage{amssymb}
\usepackage[colorlinks=true]{hyperref}

\numberwithin{equation}{section}

\newtheorem{theorem}[equation]{Theorem} 
\newtheorem*{theorem*}{Theorem}
\newtheorem{lemma}[equation]{Lemma}
\newtheorem{proposition}[equation]{Proposition}
\newtheorem{corollary}[equation]{Corollary}
\newtheorem*{corollary*}{Corollary}

\theoremstyle{remark}

\newtheorem{example}[equation]{Example}

\newtheorem{notation}[equation]{Notation}

\theoremstyle{remark}
\newtheorem{remark}[equation]{Remark}

\newcommand{\ko}{\: , \;}
\newcommand{\cA}{{\mathcal A}}
\newcommand{\cB}{{\mathcal B}}
\newcommand{\cC}{{\mathcal C}}
\newcommand{\cD}{{\mathcal D}}

\newcommand{\cM}{{\mathcal M}}

\newcommand{\cO}{{\mathcal O}}

\newcommand{\cS}{{\mathcal S}}
\newcommand{\cT}{{\mathcal T}}

\newcommand{\Spt}{\mathsf{Sp}^\Sigma}

\newcommand{\add}{\mathsf{add}}
\newcommand{\loc}{\mathsf{loc}}

\newcommand{\bfA}{\mathbf{A}}

\newcommand{\bbD}{\mathbb{D}}

\newcommand{\bbF}{\mathbb{F}}

\newcommand{\bfL}{\mathbf{L}}

\newcommand{\bbS}{\mathbb{S}}

\newcommand{\bbQ}{\mathbb{Q}}
\newcommand{\bbZ}{\mathbb{Z}}

\DeclareMathOperator{\hocolim}{hocolim}

\DeclareMathOperator{\Id}{Id}
\DeclareMathOperator{\id}{id}
\DeclareMathOperator{\incl}{incl}

\DeclareMathOperator{\Mot}{Mot}

\DeclareMathOperator{\Fun}{Fun} 

\newcommand{\Sp}{\mathsf{Sp}} 

\newcommand{\bbK}{I\mspace{-6.mu}K}

\newcommand{\dgcat}{\mathsf{dgcat}}

\newcommand{\perf}{\mathrm{perf}}

\newcommand{\dgcatf}{\dgcat_{\mathsf{f}}}

\newcommand{\dg}{\mathsf{dg}}

\newcommand{\uHom}{\underline{\mathrm{Hom}}}
\newcommand{\Hom}{\mathrm{Hom}}
\newcommand{\Nat}{\mathrm{Nat}}

\newcommand{\HomC}{\uHom_{\,!}}

\newcommand{\rep}{\mathrm{rep}}

\newcommand{\Cat}{\mathsf{Cat}}
\newcommand{\CAT}{\mathsf{CAT}}
\newcommand{\dgHo}{\mathsf{H}^0}

\newcommand{\Ho}{\mathsf{Ho}}
\newcommand{\HO}{\mathsf{HO}}

\newcommand{\op}{\mathrm{op}}

\newcommand{\sSet}{\mathsf{sSet}}

\newcommand{\Map}{\mathrm{Map}}

\newcommand{\too}{\longrightarrow}

\newcommand{\dgS}{\cS}
\newcommand{\dgD}{\cD}

\newcommand{\ie}{\textsl{i.e.}\ }

\newcommand{\Madd}{\Mot_{\mathsf{add}}} 

\newcommand{\Mloc}{\Mot_{\mathsf{loc}}} 

\begin{document}

\title[${\bfA\!}^1$-homotopy theory of noncommutative motives]{${\bf A\!}^1$-homotopy theory of noncommutative motives}

\author{Gon{\c c}alo~Tabuada}

\address{Gon{\c c}alo Tabuada, Department of Mathematics, MIT, Cambridge, MA 02139, USA}
\email{tabuada@math.mit.edu}
\urladdr{http://math.mit.edu/~tabuada/}
\thanks{The author was partially supported by a NSF CAREER Award.}
\date{\today}


\abstract{In this article we continue the development of a theory of noncommutative motives, initiated in \cite{Additive}. We construct categories of ${\bf A\!}^1$-homotopy noncommutative motives, describe their universal properties, and compute their spectra of morphisms in terms of Karoubi-Villamayor's $K$-theory ($KV$) and Weibel's homotopy $K$-theory ($KH$). As an application, we obtain a complete classification of all the natural transformations defined on $KV, KH$. This leads to a streamlined construction of Weibel's homotopy Chern character from $KV$ to periodic cyclic homology. Along the way we extend Dwyer-Friedlander's {\'e}tale $K$-theory to the noncommutative world, and develop the universal procedure of forcing a functor to preserve filtered homotopy colimits.
}}

\maketitle 
\vskip-\baselineskip

\section{Introduction}
A {\em differential graded (=dg) category}, over a base commutative ring $k$, is a category enriched over complexes of $k$-modules; see \S\ref{sec:dg}. Every (dg) $k$-algebra $A$ gives naturally rise to a dg category with a single object. Another source of examples is provided by schemes since the derived category of perfect complexes $\perf(X)$ of every quasi-compact quasi-separated $k$-scheme $X$ admits a canonical dg enhancement $\perf_\dg(X)$; see Keller \cite[\S4.6]{ICM-Keller}. As explained in \S\ref{sec:dg}, the category $\dgcat$ of (small) dg categories carries a Quillen model structure. Consequently, we obtain a well-defined Grothendieck derivator $\HO(\dgcat)$; consult Appendix \ref{sub:derivators}. A morphism of  derivators $E:\HO(\dgcat) \to \bbD$, with values in a triangulated derivator, is called:
\begin{itemize}
\item[(i)] {\em ${\bf A\!}^1$-homotopy invariant} if it inverts the dg functors $\cA \to \cA[t]:=\cA \otimes k[t]$;
\item[(ii)] {\em Additive} if it preserves filtered homotopy colimits and sends split short exact sequences of dg categories (see \cite[\S13]{Additive}) to direct sums
\begin{eqnarray*}
\xymatrix@C=1.7em{
0 \ar[r] & \cA \ar[r] & \cB \ar@/_0.5pc/[l]
\ar[r] & \cC \ar@/_0.5pc/[l] \ar[r]& 0
} & \mapsto & E(\cA) \oplus E(\cC) \simeq E(\cB)\,;
\end{eqnarray*}
\item[(iii)] {\em Localizing} if it preserves filtered homotopy colimits and sends short exact sequences of dg categories (see \cite[\S9]{Additive}) to distinguished triangles
\begin{eqnarray*}
0 \too \cA \too \cB \too \cC \too 0 & \mapsto & E(\cA) \to E(\cB) \to E(\cC) \to \Sigma E(\cA)\,.
\end{eqnarray*}
\end{itemize}
Clearly (iii) $\Rightarrow$ (ii). When $E$ satisfies (i)-(ii), resp. (i) and (iii), we call it an {\em ${\bf A\!}^1$-additive invariant}, resp. an {\em ${\bf A\!}^1$-localizing invariant}. Here are some examples:
\begin{example}{(Karoubi-Villamayor's $K$-theory)}
Karoubi and Villamayor introduced in \cite{KV,KV1} the algebraic $K$-theory groups $KV_n, n \geq 1$, of rings. In \S\ref{sub:KV} we construct the spectral enhancement $KV$ of these groups as well as its mod-$l$ variant $KV(-;\bbZ/l)$. These are examples of ${\bf A\!}^1$-additive invariants.
\end{example}
\begin{example}{(Weibel's homotopy $K$-theory)}
Weibel introduced in \cite{Weibel-KH} the algebraic $K$-theory groups $KH_n, n \in \bbZ$, of rings and schemes. In \S\ref{sub:KH} we extend these constructions to dg categories and introduce also the mod-$l$ variant $KH(-;\bbZ/l)$. These are examples of ${\bf A\!}^1$-localizing invariants.
\end{example}
\begin{example}{(Dwyer-Friedlander's {\'e}tale $K$-theory)}
Dwyer and Friedlander introduced in \cite{etale1,etale2} (see also \cite{Friedlander1,Friedlander2}) the {\'e}tale $K$-theory of schemes. In \S\ref{sub:etale}, making use of Thomason's work \cite{Thomason}, we extend this construction to (the noncommutative setting of) dg categories. This is an example of an ${\bf A\!}^1$-localizing invariant.
\end{example}
\begin{example}{(Periodic cyclic homology)}
Goodwillie (resp. Weibel) introduced in \cite{Goodwillie} (resp. in \cite{Weibel-cyclic}) the periodic cyclic homology of rings (resp. of schemes). In \S\ref{sec:HP} we extend these constructions to dg categories. As proved in Proposition~\ref{prop:HP-invariant}, the morphism of derivators obtained $HP:\HO(\dgcat) \to \HO(\Sp)$ (with values in spectra) is ${\bf A\!}^1$-homotopy invariant whenever $k$ is a field of characteristic zero. However, since periodic cyclic homology is defined using infinite products, $HP$ does {\em not} preserve filtered homotopy colimits. Consequently, $HP$ is {\em not} an ${\bf A\!}^1$-additive invariant. Making use of a universal construction of independent interest (see Proposition~\ref{prop:general}), we obtain nevertheless an ${\bf A\!}^1$-additive invariant $HP^{\mathrm{flt}}$ and a $2$-morphism $\epsilon: HP^{\mathrm{flt}} \Rightarrow HP$ whose evaluation at every  homotopically finitely presented dg category (see \S\ref{sub:smooth}) is an isomorphism. 
\end{example}
In this article we study the above properties (i)-(iii) from a motivic viewpoint.
\section{Statement of results}
\begin{theorem}\label{thm:main}
There exist morphisms of derivators 
\begin{eqnarray*}
U^{{\bf A\!}^1}_\add: \HO(\dgcat) \too \Mot^{{\bf A\!}^1}_\add && U^{{\bf A\!}^1}_\loc: \HO(\dgcat) \too \Mot^{{\bf A\!}^1}_\loc 
\end{eqnarray*} 
characterized by the following universal property: given any triangulated derivator $\bbD$ one has induced equivalences
\begin{eqnarray}
(U^{{\bf A\!}^1}_\add)^{\ast}:\ \HomC(\Mot^{{\bf A\!}^1}_\add, \bbD) & \stackrel{\sim}{\too} & \uHom_{\add,{\bf A\!}^1}(\HO(\dgcat), \bbD) \label{eq:equiv-UAdd-main} \\
(U^{{\bf A\!}^1}_\loc)^{\ast}:\ \HomC(\Mot^{{\bf A\!}^1}_\loc, \bbD) & \stackrel{\sim}{\too} & \uHom_{\loc,{\bf A\!}^1}(\HO(\dgcat), \bbD)\,, \label{eq:equiv-ULoc-main}
\end{eqnarray}
where the left-hand-sides denote the categories of homotopy colimit preserving morphisms of derivators and the right-hand-sides the categories of ${\bf A\!}^1$-additive/localizing invariants. Moreover, $\Madd^{{\bf A\!}^1}$ (resp. $\Mloc^{{\bf A\!}^1}$) carries an homotopy colimit preserving closed symmetric monoidal structure which makes $U_\add^{{\bf A\!}^1}$ (resp. $U_\loc^{{\bf A\!}^1}$) symmetric monoidal and which gives rise to a $\otimes$-enhancement of \eqref{eq:equiv-UAdd-main} (resp. of \eqref{eq:equiv-ULoc-main}). 
\end{theorem}
Roughly speaking, Theorem~\ref{thm:main} shows that an ${\bf A\!}^1$-additive (resp. ${\bf A\!}^1$-localizing) invariant is the same data as an homotopy colimit preserving morphism of derivators defined on $\Madd^{{\bf A\!}^1}$ (resp. $\Mloc^{{\bf A}^1}$). Because of these universal properties, which are reminiscent from motives, the base categories of $\Mot_\add^{{\bf A\!}^1}$ and $\Mot_\loc^{{\bf A\!}^1}$ (see Appendix \ref{sub:derivators}) are called the triangulated categories of {\em ${\bf A\!}^1$-homotopy noncommutative motives}.

Given an object $\cO$ in a triangulated category $\cT$ and an integer $l \geq 2$, let $\cdot l$ be the $l$-fold multiple of the identity of $\cO$ and $l\backslash \cO$ the fiber of $\cdot l$. As any triangulated derivator, $\Mot_\add^{\bfA^{\!1}}$ and $\Mot_\loc^{\bfA^{\!1}}$ are naturally enriched $\Hom_\Sp(-,-)$ over spectra.
\begin{theorem}\label{thm:main2}
Let $\cA$ and $\cB$ be two dg categories, with $\cA$ smooth and proper (see \S\ref{sub:smooth}). Under these assumptions, we have the following weak equivalences of spectra
\begin{eqnarray}
\Hom_\Sp(U_\add^{{\bf A\!}^1}(\cA), U_\add^{{\bf A\!}^1}(\cB))  &\simeq&  KV(\cA^\op \otimes^\bfL \cB) \nonumber\\ 
\Hom_\Sp(l \backslash U_\add^{{\bf A\!}^1}(\cA), U_\add^{{\bf A\!}^1}(\cB))  &\simeq& KV(\cA^\op \otimes^\bfL\cB;\bbZ/l) \nonumber\\
\Hom_\Sp(U_\loc^{{\bf A\!}^1}(\cA), U_\loc^{{\bf A\!}^1}(\cB)) & \simeq &  KH(\cA^\op \otimes^\bfL \cB) \label{eq:star-3}\\
\Hom_\Sp(l\backslash U_\loc^{{\bf A\!}^1}(\cA), U_\loc^{{\bf A\!}^1}(\cB))  & \simeq & KH(\cA^\op \otimes^\bfL \cB;\bbZ/l) \label{eq:star-4}\,.
\end{eqnarray}
\end{theorem}
Note that the left-hand-sides of Theorem~\ref{thm:main2} are defined solely in terms of universal properties (algebraic $K$-theory is never mentioned). Therefore, Theorem~\ref{thm:main2} provides a simple conceptual characterization of Karoubi-Villamayor and Weibel's $K$-theories. Roughly speaking, these $K$-theories are the functors co-represented by the $\otimes$-unit of the categories of ${\bf A\!}^1$-homotopy noncommutative motives. Note also that Theorem~\ref{thm:main2} combined with Theorem \ref{thm:main} implies that $\Madd^{{\bf A\!}^1}$ (resp. $\Mloc^{{\bf A\!}^1}$) is enriched over $KV(k)$-modules (resp. $KH(k)$-modules).
\begin{corollary}\label{cor:main}
Let $X$ and $Y$ be quasi-compact quasi-separated $k$-schemes, with $X$ smooth and proper, and $Y$ (or $X$) $k$-flat. Under these assumptions, we have
\begin{equation*}
\Hom_\Sp(U_\loc^{{\bf A\!}^1}(\perf_\dg(X)),U_\loc^{{\bf A\!}^1}(\perf_\dg(Y)))  \simeq KH(X\times Y)\,.
\end{equation*}
\end{corollary}
\section{Applications}
Our main application is the following (complete) classification result:
\begin{theorem}\label{thm:main22}
Given any ${\bf A\!}^1$-additive invariant $E$, with values in $\HO(\Sp)$, one has 
\begin{eqnarray}\label{eq:nat-1}
\mathrm{Nat}_\Sp(KV,E) \simeq E(k) &\mathrm{and}& \mathrm{Nat}(KV,E) \simeq E_0(k)\,,
\end{eqnarray}
where $\mathrm{Nat}_\Sp$ is the spectrum of natural transformations and $\Nat:= \pi_0\Nat_\Sp$. The same holds for ${\bf A\!}^1$-localizing invariants $E$ when $KV$ is replaced by $KH$.
\end{theorem}
Note that Theorem~\ref{thm:main22} provides a streamlined construction of natural transformations: given your favorite ${\bf A\!}^1$-additive invariant $E$, the choice of an element of $E_0(k)$ gives automatically rise to a well-defined natural transformation $KV \Rightarrow E$ ! In the particular case of periodic cyclic homology ($E=HP^{\mathrm{flt}}$) we have  
$$ \mathrm{Nat}(KV,HP^{\mathrm{flt}}) \simeq HP_0^{\mathrm{flt}}(k) \simeq HP_0(k) \simeq k\,.$$
Let us denote by $KV \Rightarrow HP^{\mathrm{flt}}$ the natural transformation corresponding to $1 \in k$ and by $ch^{{\bf A\!}^1}$ the composition $KV\Rightarrow HP^{\mathrm{flt}} \stackrel{\epsilon}{\Rightarrow} HP$. Given a dg category $\cA$, we hence obtain induced homomorphisms
\begin{eqnarray}\label{eq:Chern-A1}
ch^{{\bf A\!}^1}_n(\cA): KV_n(\cA) \too HP_n(\cA) && n \geq 0\,.
\end{eqnarray}
\begin{theorem}\label{thm:main3}
When $\cA=A$, with $A$ a $k$-algebra, the above homomorphisms \eqref{eq:Chern-A1} (with $n \geq 1$) agree with Weibel's homotopy Chern characters \cite[\S5]{Weibel-Nil}.
\end{theorem}
Theorem~\ref{thm:main3} provides a simple conceptual characterization of Weibel's homotopy Chern characters. Intuitively speaking, these are the natural transformations corresponding to the unit $1$ of the base ring $k$.
\subsection*{Acknowledgments:} 
The author is very grateful to Joseph Ayoub, Alexander Beilinson, Christian Haesemeyer, Max Karoubi, Yuri Manin, Michel Van den Bergh, and Mariuz Wodzicki for useful discussions. He would like also to thank the MSRI, Berkeley, for its hospitality and excellent working conditions.
\subsection*{Notations}
Throughout the article we will work over a base commutative ring $k$. We will use freely the language of Quillen model categories; see \cite{Hirschhorn,Hovey,Quillen}. Given a Quillen model category $\cM$, we will write $\Ho(\cM)$ for its homotopy category. The category of simplicial sets (endowed with the classical Quillen model structure \cite{Goerss}) will be denoted by $\sSet$, the category of spectra (endowed with Bousfield-Friedlander's Quillen model structure \cite{BF}) will be denoted by $\Sp$, and the category of symmetric spectra (endowed with Hovery-Shipley-Smith's stable Quillen model structure \cite{Shipley}) will be denoted by $\Spt$. Finally, adjunctions will be displayed vertically with the left (resp. right) adjoint on the left (resp. right) hand-side.
\section{Differential graded categories}\label{sec:dg}
Let $\cC(k)$ be the category of complexes of $k$-modules. A {\em differential graded (=dg) category} $\cA$ is a category enriched over $\cC(k)$. A {\em dg functor} $F:\cA\to \cB$ is  a functor enriched over $\cC(k)$; consult Keller's ICM survey \cite{ICM-Keller} for details. In what follows, we will write $\dgcat$ for the category of (small) dg categories and dg functors.

Let $\cA$ be a dg category. The category $\dgHo(\cA)$ has the same objects as $\cA$ and $\dgHo(\cA)(x,y):=H^0\cA(x,y)$. The {\em opposite} dg category $\cA^\op$ has the same objects as $\cA$ and $\cA^\op(x,y):=\cA(y,x)$. A {\em right $\cA$-module} is a dg functor $\cA^\op \to \cC_\dg(k)$ with values in the dg category $\cC_\dg(k)$ of complexes of $k$-modules. Let us write $\cC(\cA)$ for the category of right $\cA$-modules. As explained in \cite[\S3.1]{ICM-Keller}, the dg structure of $\cC_\dg(k)$ makes $\cC(\cA)$ into a dg category $\cC_\dg(\cA)$. The {\em derived category} $\cD(\cA)$ of $\cA$ is the localization of $\cC(\cA)$ with respect to quasi-isomorphisms. Its subcategory of compact objects will be denoted by $\cD_c(\cA)$.

A dg functor $F:\cA\to \cB$ is called a {\em Morita equivalence} if the restriction of scalars $\cD(\cB) \stackrel{\sim}{\to} \cD(\cA)$ is an equivalence. As proved in \cite[Theorem~5.3]{IMRN}, $\dgcat$ admits a Quillen model structure whose weak equivalences are the Morita equivalences.

The {\em tensor product $\cA\otimes\cB$} of dg categories is defined as follows: the set of objects is the cartesian product of the sets of objects of $\cA$ and $\cB$ and $(\cA\otimes\cB)((x,w),(y,z)):= \cA(x,y) \otimes \cB(w,z)$. As explained in \cite[\S2.3]{ICM-Keller}, this construction gives rise to symmetric monoidal categories $(\dgcat, -\otimes-, k)$ and $(\Ho(\dgcat), -\otimes^\bfL-, k)$.

Given dg categories $\cA$ and $\cB$, an {\em $\cA\text{-}\cB$-bimodule $\mathsf{B}$} is a dg functor $\mathsf{B}:\cA \otimes^\bfL \cB^\op\to \cC_\dg(k)$, \ie a right $(\cA^\op \otimes^\bfL \cB)$-module. A standard example is the $\cA\text{-}\cA$-bimodule
\begin{eqnarray}\label{eq:bimodule-Id}
\cA \otimes^\bfL \cA^\op \too \cC_\dg(k) && (x,y) \mapsto \cA(y,x)\,.
\end{eqnarray}
\begin{notation}\label{not:rep}
Given dg categories $\cA$ and $\cB$, let $\rep(\cA,\cB)$ be the full triangulated subcategory of $\cD(\cA^\op \otimes^\bfL \cB)$ consisting of those $\cA\text{-}\cB$-bimodules $\mathsf{B}$ such that $\mathsf{B}(x,-) \in \cD_c(\cB)$ for every object $x \in \cA$. In the same vein, let $\rep_\dg(\cA,\cB)$ be the full dg subcategory of $\cC_\dg(\cA^\op \otimes^\bfL \cB)$ consisting of those $\cA\text{-}\cB$-bimodules $\mathsf{B}$ which belong to $\rep(\cA,\cB)$. By construction, we have $\dgHo(\rep_\dg(\cA,\cB))\simeq \rep(\cA,\cB)$.
\end{notation}
\subsection{Finite dg cells}\label{sub:cells}
For $n \in \mathbb{Z}$, let $S^{n}$ be the complex $k[n]$
(with $k$ concentrated in degree $n$) and $D^n$ the mapping
cone of the identity on $S^{n-1}$. Let $\dgS(n)$ be the dg
category with two objects $1$ and $2$ such that $ \dgS(n)(1,1)=k \ko
\dgS(n)(2,2)=k \ko \dgS(n)(2,1)=0  \ko \dgS(n)(1,2)=S^{n} $ and with composition given by multiplication. Similarly, let $\dgD(n)$ be the dg
category with two objects $3$ and $4$ such that $ \dgD(n)(3,3)=k \ko
\dgD(n)(4,4)=k \ko \dgD(n)(4,3)=0 \ko \dgD(n)(3,4)=D^n $. For $n \in \bbZ$, let $\iota(n):\dgS(n-1)\to \dgD(n)$ be the dg functor that sends $1$ to
$3$, $2$ to $4$ and $S^{n-1}$ to $D^n$ by the identity on $k$ in
degree~$n-1$\,:
$$
\vcenter{
\xymatrix@C=2em@R=1em{
\dgS(n-1) \ar@{=}[d] \ar[rrr]^{\displaystyle \iota(n)}
&&& \dgD(n) \ar@{=}[d]
\\
&&&\\
\\
1 \ar@(ul,ur)[]^{k} \ar[dd]_-{S^{n-1}}
& \ar@{|->}[r] &
& 3\ar@(ul,ur)[]^{k}  \ar[dd]^-{D^n}
\\
& \ar[r]^-{\incl} &
\\
2 \ar@(dr,dl)[]^{k}
& \ar@{|->}[r] &
& 4\ar@(dr,dl)[]^{k}
}}
\qquad\text{where}\qquad
\vcenter{\xymatrix@R=1em@C=.8em{ S^{n-1} \ar[rr]^-{\incl} \ar@{=}[d]
&& D^n \ar@{=}[d]
\\
\ar@{.}[d]
&& \ar@{.}[d]
\\
0 \ar[rr] \ar[d]
&& 0 \ar[d]
\\
0 \ar[rr] \ar[d]
&& k \ar[d]^{\id}
\\
k \ar[rr]^{\id} \ar[d]
&& k \ar[d]
&{\scriptstyle(\textrm{degree }n-1)}
\\
0 \ar[rr] \ar@{.}[d]
&& 0 \ar@{.}[d]
\\
&&}}
$$
A dg category $\cA$ is called a \emph{finite dg cell} if the unique dg functor $\emptyset\to\cA$ (where the empty dg category $\emptyset$ is the initial object in $\dgcat$) can be expressed as a finite composition of pushouts along the dg functors $\iota(n), n \in \bbZ$, and $\emptyset \to k$.

\subsection{Smooth, proper, and homotopically finitely presented dg categories}\label{sub:smooth}
Recall from \cite[Definition~17.4.1]{Hirschhorn} that every Quillen model category comes equipped with a mapping space $\Map(-,-)$. A dg category $\cA$ is called {\em homotopically finitely presented} if for each filtered direct system $\{\cB_j\}_{j \in J}$ the induced map
$$ \hocolim_j \Map(\cA,\cB_j) \too \Map(\cA,\hocolim_j \cB_j)$$
is a weak equivalence of simplicial sets. As proved in \cite[Proposition~5.2]{Additive}, the homotopically finitely presented dg categories are precisely the retracts in the homotopy category $\Ho(\dgcat)$ of the finite dg cells. Recall from Kontsevich \cite{finMot,IAS,Miami} that a dg category $\cA$ is called {\em smooth} if the $\cA\text{-}\cA$-bimodule \eqref{eq:bimodule-Id} belongs to $\cD_c(\cA^\op \otimes^\bfL \cA)$ and {\em proper} if for each pair of objects $(x,y)$ we have $\sum_i \mathrm{rank}\, H^i \cA(x,y) < \infty$. The standard examples are the finite dimensional $k$-algebras of finite global dimension (when $k$ is a perfect field) and the dg categories $\perf_\dg(X)$ associated to smooth and proper $k$-schemes $X$. As proved in \cite[Proposition~5.10]{CT1}, every smooth and proper dg category is homotopically finitely presented. 
\section{Algebraic $K$-theories}\label{sec:K-theories}
Let $k[t]$ be the $k$-algebra of polynomials and 
\begin{eqnarray}\label{eq:homot1}
\iota: k \hookrightarrow k[t] && ev_0, ev_1: k[t] \to k
\end{eqnarray}
the inclusion and evaluation maps. Given a dg category $\cA$, let $\iota: \cA \to \cA[t]$ and $ev_0, ev_1: \cA[t] \to \cA$ be the dg functors obtained by tensoring $\cA$ with \eqref{eq:homot1}. 
\subsection{$\bfA^{\!1}$-homotopization}
Let $\cM$ be a model category, $E:\dgcat \to \cM$ a functor sending Morita equivalences to weak equivalences, $E:\HO(\dgcat) \to \HO(\cM)$ the associated morphism of derivators, and $
\Delta_n:= k[t_0, \ldots, t_n]/ (\sum_{i=0}^n t_i -1), n \geq 0$, the simplicial $k$-algebra with faces and degenerancies given by the formulas
\begin{eqnarray*}
d_r(t_i) := \left\{ \begin{array}{lcr}
t_i & \text{if} & i <r \\
0 & \text{if} & i =r \\
t_{i-1} & \text{if} & i > r \\
\end{array} \right.
&
&
s_r(t_i) := \left\{ \begin{array}{lcr}
t_i & \text{if} & i <r \\
t_i + t_{i+1} & \text{if} & i =r \\
t_{i+1} & \text{if} & i > r \\
\end{array} \right.\,.
\end{eqnarray*}
Out of this data, one constructs the {\em ${\bf A\!}^1$-homotopization} of $E$:
\begin{eqnarray*}
E^h:\HO(\dgcat) \too \HO(\cM) & & \cA \mapsto \mathrm{hocolim}_n\, E(\cA \otimes \Delta_n)\,.
\end{eqnarray*}
Note that $E^h$ comes equipped with a $2$-morphism $\eta: E \Rightarrow E^h$.

\begin{proposition}\label{prop:A1-homot}
\begin{itemize}
\item[(i)] The morphism $E^h$ is $\bfA^{\!1}$-homotopy invariant.
\item[(ii)] When $E$ is $\bfA^{\!1}$-homotopy invariant, $\eta: E \Rightarrow E^h$ is a $2$-isomorphism.
\item[(iii)] When $E$ is additive/localizing, $E^h$ is also additive/localizing.
\item[(iv)] When $\cM$ carries an homotopy colimit preserving symmetric monoidal structure and $E$ is symmetric monoidal, $E^h$ is also symmetric monoidal.
\end{itemize}
\end{proposition}
\begin{proof}
One one hand we have $ev_0 \circ \iota =\id$. On the other hand, the simplicial map
$(k[t] \stackrel{ev_0}{\to} k \stackrel{\iota}{\to} k[t]) \otimes \Delta_n, n \geq 0$, is homotopic to $\id$ via the simplicial homotopy 
\begin{equation}\label{eq:simplicial-homotopy}
\{h_j: k[t] \otimes \Delta_n \too k[t] \otimes \Delta_{n+1} \}_{0 \leq j \leq n}
\end{equation}
that sends $t \mapsto t (t_{j+1} + \cdots + t_{n+1})$ and $t_i \mapsto s_j(t_i)$.  
By first tensoring $\cA$ with \eqref{eq:simplicial-homotopy} and then by applying the functors $E:\dgcat \to \cM$ and $\mathrm{hocolim}_n:\HO(\cM)(\Delta) \to \Ho(\cM)$ (where $\Delta$ is the category of finite ordinal numbers with order-preserving maps between them), we conclude that $E^h(\iota \circ ev_0)=\id$. This implies that the map
$$ E^h(\cA) := \mathrm{hocolim}_n E(\cA\otimes \Delta_n) \too \mathrm{hocolim}_n E(\cA \otimes k[t] \otimes \Delta_n)=:E^h(\cA[t])$$
is an isomorphism and so item (i) is proved. Item (ii) follows from the fact that all the maps of the simplicial object $n \mapsto E(\cA \otimes \Delta_n)$ are isomorphisms whenever $E$ is ${\bf A\!}^1$-homotopy invariant. In what concerns item (iii), note first that $\Delta_0\simeq k$ and $\Delta_n \simeq k[t_0, \cdots, t_{n-1}]$ for $n >0$. This implies that the $k$-algebras $\Delta_n, n \geq 0$, are flat. As a consequence, we obtain well-defined morphisms of derivators
\begin{eqnarray}\label{eq:Delta}
- \otimes \Delta_n : \HO(\dgcat) \too \HO(\dgcat) && n \geq 0\,.
\end{eqnarray}
Thanks to Drinfeld \cite[Proposition~1.6.3]{Drinfeld}, these morphisms preserve (split) short exact sequences of dg categories. Moreover, since the symmetric monoidal structure on $\HO(\dgcat)$ is homotopy colimit preserving (see \cite[Proposition~3.3]{CT1}), the morphisms \eqref{eq:Delta} preserve also filtered homotopy colimits. These facts imply item (iii). Finally, item (iv) follows from the following sequence of isomorphisms
\begin{eqnarray}
E^h(\cA) \otimes E^h(\cB) & := & \mathrm{hocolim}_n E(\cA \otimes \Delta_n) \otimes \mathrm{hocolim}_{n'} E(\cB \otimes \Delta_{n'}) \nonumber \\ 
& \simeq & \mathrm{hocolim}_{n,n'} (E(\cA \otimes \Delta_n) \otimes E(\cB \otimes \Delta_n))\label{eq:star-1} \\
& \simeq & \mathrm{hocolim}_{n,n'} E((\cA\otimes^\bfL \cB) \otimes (\Delta_n \otimes \Delta_{n'})) \label{eq:star-2} \\
& \simeq & \mathrm{hocolim}_n E(\cA\otimes^\bfL \cB \otimes \Delta_n) =: E^h(\cA \otimes^\bfL \cB)\,. \label{eq:star-33}
\end{eqnarray}
Some explanations are in order: \eqref{eq:star-1} follows from the assumption that the symmetric monoidal structure on $\cM$ is homotopy colomit preserving; \eqref{eq:star-2} follows from the fact that $E$ is symmetric monoidal; and \eqref{eq:star-33} follows from the cofinality of the diagonal map $\Delta \to \Delta \times \Delta$.
\end{proof}
\begin{remark}\label{rk:spectral}
When $\cM$ is the Quillen model category of spectra $\Sp$, one has a standard convergent right half-plane spectral sequence $E^1_{pq}=N^p \pi_q E(\cA) \Rightarrow \pi_{p+q} E^h(\cA)$, where $N^\ast \pi_q E(\cA)$ is the Moore complex of the simplical group $n \mapsto \pi_qE(\cA\otimes \Delta_n)$.
\end{remark}
\subsection{Karoubi-Villamayor's $K$-theory}\label{sub:KV}
Recall from \cite[Example~15.6]{Additive} the construction of connective algebraic $K$-theory $K:\HO(\dgcat) \to \HO(\Sp)$. This additive invariant is induced from a functor $\dgcat \to \Sp$ (sending Morita equivalences to weak equivalences) and so thanks to Proposition~\ref{prop:A1-homot} it gives rise to a well-defined ${\bf A\!}^1$-additive invariant
\begin{eqnarray*}
KV:= K^h:\HO(\dgcat) \too \HO(\Sp) && \cA \mapsto \mathrm{hocolim}_n K(\cA\otimes \Delta_n)\,.
\end{eqnarray*}
Remark~\ref{rk:spectral} furnishes a convergent spectral sequence $E^1_{p,q}=N^p K_q(\cA) \Rightarrow KV_{p+q} (\cA)$.
\begin{proposition}[Agreement]\label{prop:agreement}
When $\cA=A$, with $A$ a $k$-algebra, the groups $KV_n(\cA), n \geq 1$, agree with the Karoubi-Villamayor's $K$-theory groups of $A$.
\end{proposition}
\begin{proof}
Let $KV(A)\langle 0\rangle$ be the $0$-connected cover of $KV(A)$. As explained in \cite[\S IV page 83]{Weibel-book}, the Karoubi-Villamayor's algebraic $K$-theory groups $KV_n(A), n \geq 1$, agree with the homotopy groups of $KV(A)\langle 0\rangle$. This implies our claim.
\end{proof}
\begin{notation}
Let $\cO$ be an object in a triangulated category $\cT$ and $l \geq 2$ an integer. We define the {\em mod-$l$ Moore object} $\cO/l$ of $\cO$ as the cofiber of $\cdot l:\cO \to \cO$.
\end{notation}
Given $l \geq 2$, consider the {\em mod-$l$ Karoubi-Villamayor's algebraic $K$-theory}
\begin{eqnarray*}
KV(-;\bbZ/l):\HO(\dgcat) \to \HO(\Sp) && \cA \mapsto KV(\cA) \wedge^\bfL \bbS/l\,,
\end{eqnarray*}
where $\bbS/l$ is the mod-$l$ Moore spectrum of $\bbS$. Since $-\wedge^\bfL \bbS/l$ preserves direct sums, $KV(-;\bbZ/l)$ is also an ${\bf A\!}^1$-additive invariant. Moreover, thanks to the universal coefficients theorem (see \cite[\S IV page~19]{Weibel-book}), we have the short exact sequence
$$ 0 \to KV_n(\cA)\otimes_\bbZ \bbZ/l \to KV_n(\cA;\bbZ/l) \to \{l\text{-}\mathrm{torsion}\,\,\mathrm{in}\,\,KV_{n-1}(\cA)\} \to 0\,.$$

\subsection{Weibel's homotopy $K$-theory}\label{sub:KH}
Recall from \cite[Theorem~10.3]{Additive} the construction of nonconnective algebraic $K$-theory 
$\bbK: \HO(\dgcat) \to \HO(\Sp)$. This localizing invariant is induced from a functor $\dgcat \to \Sp$ (sending Morita equivalences to weak equivalences) and so thanks to Proposition~\ref{prop:A1-homot} it gives rise to a well-defined $\bfA^{\!1}$-localizing invariant
\begin{eqnarray*}
KH:=\bbK^h: \HO(\dgcat) \too \HO(\Sp) && \cA \mapsto \mathrm{hocolim}_n\, \bbK(\cA \otimes \Delta_n)\,.
\end{eqnarray*}
Remark \ref{rk:spectral} furnishes a convergent spectral sequence $E^1_{p,q}=N^p \bbK_q(\cA) \Rightarrow KH_{p+q} (\cA)$. Given an integer $l \geq 2$, consider the {\em mod-$l$ Weibel's homotopy $K$-theory}
\begin{eqnarray*}
KH(-;\bbZ/l):\HO(\dgcat) \to \HO(\Sp) && \cA \mapsto KH(\cA) \wedge^\bfL \bbS/l\,.
\end{eqnarray*}
Since $-\wedge \bbS/l$ preserves distinguished triangles, $KH(-;\bbZ/l)$ is also an ${\bf A\!}^1$-localizing invariant. As above, we have the short exact sequence
$$ 0 \to KH_n(\cA)\otimes_\bbZ \bbZ/l \to KH_n(\cA;\bbZ/l) \to \{l\text{-}\mathrm{torsion}\,\,\mathrm{in}\,\,KH_{n-1}(\cA)\} \to 0\,.$$
\begin{proposition}[Agreement]\label{prop:agreementKH}
Let $\cA$ be a dg category.
\begin{itemize}
\item[(i)] When $\cA=A$, with $A$ a $k$-algebra, $KH(\cA)$ agrees with Weibel's homotopy algebraic $K$-theory of $A$. 
\item[(ii)]  When $\cA= \perf_\dg(X)$, with $X$ a quasi-compact quasi-separated $k$-scheme, $KH(\cA)$ agrees with Weibel's homotopy algebraic $K$-theory of $X$.
\end{itemize}
\end{proposition}
\begin{proof}
Item (i) follows automatically from Weibel's definition~\cite[Definition~1.1]{Weibel-KH} and from the natural identification $ A \otimes \Delta_n \simeq \Delta_nA$, where $\Delta_nA$ is the coordinate ring $A[t_0, \dots, t_n] / \big(\sum_{i=0}^n t_i -1\big)A$ of ``standard $n$-simplexes'' over $A$. In what concerns item (ii), we have the following weak equivalences of spectra
\begin{eqnarray}
\bbK (\perf_\dg(X) \otimes \Delta_n) & \simeq & \bbK (\perf_\dg(X) \otimes^\bfL \perf_\dg(\mathrm{Spec}(\Delta_n))) \nonumber\\
& \simeq & \bbK (\perf_\dg (X \times \mathrm{Spec} (\Delta_n))) \label{eq:star-11}\\
& \simeq & \bbK (X \times \mathrm{Spec} (\Delta_n)) \nonumber\,,
\end{eqnarray}
where \eqref{eq:star-11} follows from \cite[Proposition~8.2]{Regularity} (in {\em loc. cit.} we assumed $X$ to be separated; however the same result holds with $X$ quasi-separated) since $\mathrm{Spec}(\Delta_n)$ is flat and $\bbK$ is localizing. As a consequence, $\bbK^h(\perf_\dg(X)) \simeq \mathrm{hocolim}_n\, \bbK(X \times \mathrm{Spec}(\Delta_n))$. This latter spectrum is equivalent to the one defined by Weibel in \cite[Definition~6.5]{Weibel-KH} using \v{C}ech's cohomological descent; see Thomason-Trobaugh~\cite[\S9.11]{TT}.\end{proof}
\subsection{Dwyer-Friedlander's {\'e}tale $K$-theory}\label{sub:etale}
Let $l^\nu$ be a prime power with $l$ odd. Assume that $1/l \in k$. Let $K(1)$ be the first Morava $K$-theory spectrum and $L_{K(1)}:\HO(\Sp) \to \HO(\Sp)$ the associated left Bousfield localization; see Mitchell \cite[\S3.3]{Mitchell}. Since $L_{K(1)}$ is triangulated we have the following ${\bf A\!}^1$-localizing invariant
\begin{eqnarray*}
K^{et}(-;\bbZ/l^\nu): \HO(\dgcat) \too \HO(\Sp) && \cA \mapsto L_{K(1)} KH(\cA;\bbZ/l^\nu)\,.
\end{eqnarray*}
We call it the {\em Dwyer-Friedlander {\'e}tale $K$-theory}. This is justified as follows:
\begin{theorem}[Agreement]
Let $X$ be a quasi-compact separated $k$-scheme which is regular and of finite type over $\bbZ[1/l]$, or $\bbQ$, or $\bbF_p$ with $p \neq l$, or $\bbF_p[[t]]$ with $p \neq l$, or $\bbF_p((t))$ with $p \neq l$, or $\bbZ_p^{\wedge}$ with $p \neq l$, or $\bbQ_p^{\wedge}$, or over $\overline{k}$ a separable closed field of characteristic different from $l$. Under these assumptions, $K^{et}(\perf_\dg(X),\bbZ/l^\nu)$ agrees with Dwyer-Friedlander's {\'e}tale $K$-theory of $X$.
\end{theorem}
\begin{proof}
Since by assumption $1/l \in k$, one has $\bbK(X;\bbZ/l^\nu) \simeq KH(X;\bbZ/l^\nu)$; see \cite[Thm.~9.5]{TT}. Hence, the proof follows from Thomason's celebrated result \cite[Theorem~4.11]{Thomason}; see also \cite[Remark~4.2 and \S A.14]{Thomason}.
\end{proof}

\section{Periodic cyclic homology}\label{sec:HP}
Recall from \cite[\S8-9]{CT1} the construction of periodic cyclic homology
\begin{equation}\label{eq:HP}
HP:\HO(\dgcat) \stackrel{M}{\too} \HO(\cC(\Lambda)) \stackrel{P}{\too} \HO(k[u]\text{-}\mathrm{Comod}) \stackrel{\Hom_\Sp(k[u],-)}{\too} \HO(\Sp) \,.
\end{equation}
Same explanations are in order: $\cC(\Lambda)$ is the Quillen model category of mixed complexes; $M$ is induced by the mixed complex construction; $k[u]\text{-}\mathrm{Comod}$ is the Quillen model category of $k[u]$-comodules (where $k[u]$ is the Hopf algebra of polynomials in one variable $u$ of degree $2$); and finally $P$ is induced by the perioditization construction. When applied to $A$, respectively to $\perf_\dg(X)$, \eqref{eq:HP} agrees with Goodwillie's periodic cyclic homology of $A$, respectively with Weibel's periodic cyclic homology of $X$; see Keller \cite[Theorem~5.2]{ICM-Keller}.
\begin{proposition}\label{prop:HP-invariant}
When $k$ is a field of characteristic zero, the above morphism of derivators $HP$ is ${\bf A\!}^1$-homotopy invariant. 
\end{proposition}
\begin{proof}
Kassel's property {\it (P)} (see \cite[page~211]{Kassel}) is clearly verified by the $k$-algebras $k$ and $k[t]$. Therefore, \cite[Theorem~3.10]{Kassel} gives rise to the isomorphisms 
\begin{eqnarray*}
HP(\cA \otimes k) \simeq HP(\cA) \otimes HP(k) &&HP(\cA \otimes k[t]) \simeq HP(\cA) \otimes HP(k[t])\,.
\end{eqnarray*} 
This implies that \eqref{eq:HP} is ${\bf A\!}^1$-homotopy invariant if and only if $HP(k) \to HP(k[t])$ is an isomorphism. Since by assumption $k$ is a field of characteristic zero, Kassel's ${\bf A\!}^1$-homotopy invariance results (see \cite[Corollary 3.12 and (3.13)]{Kassel}) allow us to conclude that this is indeed the case. This achieves the proof.
\end{proof}
Since periodic cyclic homology is defined using infinite products, $HP$ does {\em not} preserve filtered homotopy colimits. The problem is that $k[u]$ is {\em not} a compact object of $\HO(k[u]\text{-}\mathrm{Comod})$. As a consequence, $HP$ is {\em not} an additive invariant. Making use of Proposition \eqref{prop:general} below we obtain nevertheless an ${\bf A\!}^1$-additive invariant (when $k$ is a field of characteristic zero)
\begin{eqnarray*}
HP^{\mathrm{flt}}: \HO(\dgcat) \too \HO(\Sp) && \cA \mapsto HP^{\mathrm{flt}}(\cA)
\end{eqnarray*}
and a $2$-morphism $\epsilon: HP^{\mathrm{flt}} \Rightarrow HP$.
\begin{proposition}\label{prop:general}
Given any derivator $\bbD$, one has an adjunction of categories
\begin{equation}\label{eq:adj-desired}
\xymatrix{
\uHom(\HO(\dgcat),\bbD) \ar@<1ex>[d]^-{(-)^{\mathrm{flt}}} \\
*+<3ex>{\uHom_{\mathsf{flt}}(\HO(\dgcat),\bbD)} \ar@{^{(}->}@<1ex>[u]}
\end{equation}
Given $E \in \uHom(\HO(\dgcat),\bbD)$, the following holds:
\begin{itemize}
\item[(i)] The evaluation of the counit $2$-morphism $\epsilon: E^{\mathrm{flt}} \Rightarrow E$ at every homotopically finitely presented dg category is an isomorphism;
\item[(ii)] When $E$ sends split short exact sequences to direct sums, $E^{\mathrm{flt}}$ is additive;
\item[(iii)] When $E$ is ${\bf A\!}^1$-homotopy invariant, $E^{\mathrm{flt}}$ is also ${\bf A\!}^1$-homotopy invariant.
\end{itemize}
\end{proposition}
\begin{proof}
We start by constructing the right adjoint $(-)^{\mathrm{flt}}$. Recall from \cite[\S5]{Additive} that we have the following diagram
$$
\xymatrix{
\dgcat_f[S^{-1}] \ar[d]_-h \ar[r]^-i & *+<2ex>{\HO(\dgcat)}  \ar@<1ex>[dl]^-{\mathsf{h}}\\
*+<3ex>{\HO(L_S\!\Fun(\dgcat_f^\op, \sSet))} \ar@<1ex>[ur]^-{\mathrm{Re}} & 
}
$$
with $\mathsf{h} \circ i \simeq h$ and $\mathrm{Re} \circ h \simeq i$. Some explanations are in order: $\dgcat_f$ is the (essentially) small subcategory of $\dgcat$ obtained by stabilizing the finite dg cells with respect to fibrant and cosimplicial cofibrant resolutions; $S$ is the set of Morita equivalences in $\dgcat_f$; $\dgcat_f[S^{-1}]$ is the associated prederivator (see \cite[\S A.1]{CT1}); $h$ is induced by the Yoneda embedding; $\Fun(\dgcat_f^\op, \sSet)$ is endowed with the projective Quillen model structure and $L_S\!\Fun(\dgcat_f^\op, \sSet)$ is its left Bousfield localization with respect to the image of $S$ under $h$; $\mathsf{h}$ is fully-faithful and preserves filtered homotopy colimits; and finally $(\mathrm{Re},\mathsf{h})$ is an adjunction. This latter adjunction gives automatically rise to the following one (with $\mathsf{h}^\ast$ fully-faithful)
\begin{equation}\label{eq:adj-1}
\xymatrix{
\uHom(\HO(\dgcat),\bbD) \ar@<1ex>[d]^-{\mathrm{Re}^\ast} \\
\uHom(\HO(L_S\!\Fun(\dgcat_f^\op, \sSet)), \bbD) \ar@<1ex>[u]^-{\mathsf{h}^\ast}\,.
}
\end{equation}
Thanks to \cite[Theorem~3.1]{Additive}, we have the induced equivalence
\begin{equation}\label{eq:equivalence-key}
h^\ast: \uHom_!(\HO(L_S\!\Fun(\dgcat_f^\op, \sSet)),\bbD) \stackrel{\sim}{\too} \uHom(\dgcat_f[S^{-1}],\bbD)\,.
\end{equation}
Moreover, \cite[Lemma~3.2]{Additive} gives rise to the following adjunction
\begin{equation}\label{eq:adj-2}
\xymatrix{
\uHom(\HO(L_S\!\Fun(\dgcat_f^\op, \sSet)), \bbD) \ar@<1ex>[d]^-{\Psi}  & E' \ar@{|->}[d]\\
*+<3ex>{\uHom_!(\HO(L_S\!\Fun(\dgcat_f^\op, \sSet)), \bbD)} \ar@{^{(}->}@<1ex>[u] & \Psi(E'):= \overline{E' \circ h}\,,
}
\end{equation}
where $\overline{E' \circ h}$ is the unique homotopy colimit preserving morphism of derivators corresponding to $E' \circ h$ under the above equivalence \eqref{eq:equivalence-key}. As proved in \cite[Theorem~5.13]{Additive}, we have also the following induced equivalence
\begin{equation}\label{eq:equivalence-key2}
\mathsf{h}^\ast: \uHom_!(\HO(L_S\!\Fun(\dgcat_f^\op, \sSet)), \bbD)  \stackrel{\sim}{\too} \uHom_{\mathsf{flt}}(\HO(\dgcat),\bbD)\,.
\end{equation}
By concatenating \eqref{eq:adj-1} with \eqref{eq:adj-2}-\eqref{eq:equivalence-key2}, one hence obtains the desired adjunction \eqref{eq:adj-desired}. Making use of $\mathrm{Re} \circ h \simeq i$, one observes that the right adjoint functor $(-)^{\mathrm{flt}}:=\mathsf{h}^\ast \circ \Psi \circ \mathrm{Re}^\ast$ sends a morphism of derivators $E:\HO(\dgcat) \to \bbD$ to $E^{\mathrm{flt}}:= \overline{E \circ i} \circ \mathsf{h}$.

We now have all the ingredients needed for the proof of items (i)-(iii). Making use of $\mathsf{h} \circ i = h$, one observes that the evaluation of the counit $2$-morphism $\epsilon: E^{\mathrm{flt}} \Rightarrow E$ at every dg category $\cA \in \dgcat_f$ is an isomorphism. Since the homotopically finitely presented dg categories are retracts (in the homotopy category $\Ho(\dgcat)$) of finite dg cells, we hence obtain item (i). As proved in \cite[Proposition~13.2]{Additive}, every split short exact sequence of dg categories is Morita equivalent to a filtered homotopy colimit of split short exact sequences whose components are finite dg cells. By combining this fact with item (i) and with the fact $E^{\mathrm{flt}}$ preserves filtered homotopy colimits, we obtain item (ii). Finally, item (iii) follows from item (i), from the fact that $E^{\mathrm{flt}}$ preserves filtered homotopy colimts, and from Lemma~\ref{lem:keyfilt} below.
\end{proof}
 \begin{lemma}\label{lem:keyfilt}
Given a dg category $\cA$, there exists a filtered direct system of finite dg cells $\{\cB_j\}_{j \in J}$ such that
\begin{equation}\label{eq:iso-A1}
\mathrm{hocolim}_j\, (\cB_j \to \cB_j[t]) \stackrel{\sim}{\too} (\cA \to \cA[t])\,.
\end{equation}
\end{lemma}
\begin{proof}
As proved in \cite[Proposition 3.6(iii)]{CT}, there exists a filtered direct system of finite dg cells $\{\cB_j\}_{j\in J}$ such that $\mathrm{hocolim}_j\,\cB_j  \simeq \cA$. Since the $k$-algebra $k[t]$ is flat, the functor $-\otimes k[t]$ preserves filtered homotopy colimits. Hence, by combining these two facts, we obtain the desired isomorphism \eqref{eq:iso-A1}.
\end{proof}
\section{Proof of Theorem~\ref{thm:main}}\label{sec:proof-main1}
We will focus ourselves in the localizing case. The proof of the additive case is similar. Recall from \cite[\S10]{Additive} the construction of the universal localizing invariant
\begin{equation*}
U_\loc : \HO(\dgcat) \too \Mloc\,.
\end{equation*}
Given any triangulated derivator $\bbD$, one has an induced equivalence of categories
\begin{equation}\label{eq:equiv-Uloc}
(U_\loc)^\ast: \HomC(\Mloc, \bbD) \stackrel{\sim}{\too} \uHom_\loc(\HO(\dgcat),\bbD)\,.
\end{equation}
\begin{remark}{(Quillen model)}\label{rk:Qmodel}
Consider the category $\Fun(\dgcat_f^\op,\Sp)$ endowed with the projective Quillen model structure; recall from the proof of Proposition~\ref{prop:general} the definition of the category $\dgcat_f$. As explained in \cite[\S10-11]{Additive}, $\Mloc$ admits a left proper cellular Quillen model $\Mloc^Q$ given by the left Bousfield localization of $\Fun(\dgcat_f^\op,\Sp)$ with respect to a set $\loc$ of morphisms which implement the localizing property. Moreover, $U_\loc$ is induced by the~functor
\begin{eqnarray*}
\dgcat \too \Mloc^Q && \cA \mapsto \big(\cB \mapsto \Sigma^\infty (N w \rep_\dg(\cB,\cA)_+) \big)\,,
\end{eqnarray*}
where $w\rep_\dg(\cB,\cA)$ stands for the category of quasi-isomorphisms of $\rep_\dg(\cB,\cA)$, $Nw\rep_\dg(\cB,\cA)$ for its nerve, and $\Sigma^\infty(-_+)$ for the suspension spectrum.
\end{remark}
Following \cite[\S A.7]{CT1}, one can consider the left Bousfield localization of $\Mloc^Q$ with respect to the following set of maps
\begin{equation*}
\mathsf{S}:=\{\Omega^n(U_\loc(\cB \to \cB[t])) \, |\, \cB \,\, \mathrm{finite} \,\, \mathrm{dg} \,\, \mathrm{cell}, n \geq 0\}\,,
\end{equation*}
where $\Omega$ stands for desuspension. Thanks to \cite[Theorem~A.4 and Proposition~A.6]{CT1}, we obtain a well-defined triangulated derivator $\Mloc^{{\bf A\!}^1}$ (admiting a Quillen model $\Mloc^{{\bf A\!}^1\!,Q}:=L_{\mathsf{S},\loc}\!\Fun(\dgcat_f^\op,\Sp)$) as well as an adjunction
$$
\xymatrix{
\Mloc \ar@<-1ex>[d]_{l_!} \\
\Mloc^{\bfA^{\!1}} \ar@<-1ex>[u]_{l^*}\,.
}
$$
The theory of left Bousfield localization (see \cite[\S A.7]{CT1}) implies that
\begin{equation}\label{eq:11}
(l_!)^\ast: \uHom_!(\Mloc^{\bfA^{\!1}},\bbD) \stackrel{\sim}{\too} \uHom_{!,\mathsf{S}}(\Mloc,\bbD)\,,
\end{equation}
where the right-hand-side denotes the category of homotopy colimit preserving morphisms of derivators which invert the elements of $\mathsf{S}$. Since $U_\loc$ preserves filtered homotopy colimits one concludes then from Lemma~\ref{lem:keyfilt} that \eqref{eq:equiv-Uloc} restricts to
\begin{equation}\label{eq:22}
(U_\loc)^\ast: \uHom_{!,S} (\Mloc,\bbD) \stackrel{\sim}{\too} \uHom_{\loc, {\bf A\!}^1}(\HO(\dgcat),\bbD)\,.
\end{equation}
Finally, by combining \eqref{eq:11}-\eqref{eq:22} we obtain the desired equivalence \eqref{eq:equiv-ULoc-main}.

Let us now prove the second claim. Recall from \cite[Theorem~8.5]{CT1} that $\Mloc$ carries an homotopy colimit preserving symmetric monoidal structure making $U_\loc$ symmetric monoidal. Given any triangulated derivator $\bbD$, endowed with an homotopy colimit preserving symmetric monoidal structure, one has an induced equivalence (which is a $\otimes$-enhancement of \eqref{eq:equiv-Uloc})
\begin{equation}\label{eq:4}
(U_\loc)^\ast: \uHom^\otimes_!(\Mloc,\bbD) \stackrel{\sim}{\too} \uHom^\otimes_{\loc}(\HO(\dgcat),\bbD)\,,
\end{equation}
where the left-hand-side denotes the category of symmetric monoidal homotopy colimit preserving morphisms of derivators and the right-hand-side the category of symmetric monoidal ${\bf A\!}^1$-localizing invariants.
\begin{remark}{(Symmetric monoidal Quillen model)}\label{rk:symmetric}
Recall from \cite[\S8.1]{CT1} the construction of the (essentially) small category $\dgcatf$. This full subcategory of $\dgcat_f$ is symmetric monoidal and every object of $\dgcat_f$ is Morita equivalence to an object in $\dgcatf$. Hence, as explained in {\em loc. cit.}, $L_{\loc}\!\Fun(\dgcatf^\op,\Spt)$ (endowed with the Day convolution product) is a symmetric monoidal Quillen model $\Mloc^{Q,\otimes}$ of $\Mloc$. Moreover, the following functor
\begin{eqnarray}\label{eq:functor-monoidal}
\dgcat \too \Mloc^{Q,\otimes} && \cA \mapsto (\cB \mapsto \Sigma^\infty (Nw \rep_\dg(\cB,\cA)_+))\,,
\end{eqnarray}
with $\Sigma^\infty(-_+)$ taking values in symmetric spectra, is symmetric monoidal. 
\end{remark}
Let us now verify that for every noncommutative motive $N$ the functor $N \otimes^\bfL - : \Mloc^{Q,\otimes} \to \Mloc^{Q,\otimes}$ sends the elements of $\mathsf{S}$ to $\mathsf{S}$-local weak equivalences. The category $\Mloc^{Q,\otimes}$ is generated by the noncommutative motives of the form $U_\loc(\cA)$, with $\cA$ a dg category, and the Day convolution product is homotopy colimit preserving. Hence, it suffices to show that the functors $U_\loc(\cA) \otimes^\bfL -$ send the elements of $\mathsf{S}$ to the $\mathsf{S}$-local weak equivalences. This is indeed the case since
$$ U_\loc(\cA) \otimes^\bfL \Omega^n\big(U_\loc(\cB \to \cB[t])\big) \simeq \Omega^n U_\loc\big((\cA\otimes^\bfL \cB) \to (\cA\otimes^\bfL\cB)[t] \big)\,. $$
Thanks to \cite[Proposition~6.6]{CT1} (recall from the proof of \cite[Theorem~8.5]{CT1} that all the remaining conditions of this proposition are already satisfied) we obtain a well-defined symmetric monoidal Quillen model category $\Mloc^{{\bf A\!}^1,Q,\otimes}$. Consequently, \cite[Propositions A.2 and A.9]{CT1} imply that $\Mloc^{{\bf A\!}^1}$ carries an homotopy colimit preserving symmetric monoidal structure, that $l_!$ is symmetric monoidal, and that we have an induced equivalence
\begin{equation}\label{eq:5}
(l_!)^\ast: \uHom_!^\otimes(\Mloc^{\bfA^{\!1}},\bbD) \stackrel{\sim}{\too} \uHom_{!,\mathsf{S}}^\otimes (\Mloc,\bbD)\,.
\end{equation}
Since $U_\loc$ is symmetric monoidal and preserves filtered homotopy colimits one concludes once again from Lemma~\ref{lem:keyfilt} that \eqref{eq:4} restricts to
\begin{equation}\label{eq:6}
(U_\loc)^\ast: \uHom_{!,\mathsf{S}}^\otimes(\Mloc,\bbD) \stackrel{\sim}{\too} \uHom_{\loc, {\bf A\!}^1}^\otimes (\HO(\dgcat),\bbD)\,.
\end{equation}
 Finally, by combining \eqref{eq:5}-\eqref{eq:6} one obtains the desired $\otimes$-enhancement of \eqref{eq:equiv-ULoc-main}
\begin{equation}\label{eq:monoidal-last}
(U^{{\bf A\!}^1}_\loc)^{\ast}:\ \uHom_!^\otimes(\Mot^{{\bf A\!}^1}_\loc, \bbD) \stackrel{\sim}{\too} \uHom^\otimes_{\loc,{\bf A\!}^1}(\HO(\dgcat), \bbD)\,.
\end{equation}
It remains only to show that the symmetric monoidal structure on $\Mloc^{{\bfA\!}^1}$ is closed. By construction, the Quillen model $\Mloc^{{\bfA\!}^1,Q,\otimes}$ is combinatorial in the sense of Smith, \ie it is cofibrantly generated and the underlying category is locally presentable. Following Rosicky \cite[Proposition~6.10]{Rosicky}, we conclude that the triangulated base category $\Mloc^{{\bfA\!}^1}(e)$ is well-generated in the sense of Neeman. Given any noncommutative motive $N$, the functor $-\otimes^\bfL N: \Mloc^{{\bfA\!}^1}(e) \to \Mloc^{{\bfA\!}^1}(e)$ is triangulated and preserves arbitrary coproducts. Hence, thanks to Neeman \cite[Theorem~8.4.4]{Neeman}, it admits a right adjoint $\mathsf{R} \Hom(N,-)$ which by definition is the internal-Hom functor. This implies that the symmetric monoidal structure is closed.
\section{Proof of Theorem~\ref{thm:main2}}\label{sec:proof-main2}
Similarly to the proof of Theorem~\ref{thm:main}, we will focus ourselves in the localizing case, \ie in the proof of weak equivalences \eqref{eq:star-3}-\eqref{eq:star-4}. As explained in Remark~\ref{rk:symmetric}, the Quillen model $\Mloc^{Q,\otimes}$ carries an homotopy colimit preserving symmetric monoidal structure and the functor \eqref{eq:functor-monoidal} is symmetric monoidal. Thanks to Proposition~\ref{prop:A1-homot}, we obtain then a well-defined symmetric monoidal ${\bf A\!}^1$-localizing invariant $U_\loc^h:\HO(\dgcat) \to \Mloc$ and a $2$-morphism $\eta: U_\loc \Rightarrow U_\loc^h$. Consequenty, equivalence \eqref{eq:monoidal-last} gives rise to a symmetric monoidal homotopy colimit preserving morphism $\overline{U_\loc^h}: \Mloc^{{\bf A\!}^1} \to \Mloc$ such that $\overline{U_\loc^h} \circ U_\loc^{{\bf A\!}^1} \simeq U_\loc^h$. The proof  of \eqref{eq:star-3} follows now from the following weak equivalences of spectra
\begin{eqnarray}
\Hom_\Sp(U_\loc^{\bfA^{\!1}}(\cA),U_\loc^{\bfA^{\!1}}(\cB)) & \simeq & \Hom_\Sp(U_\loc(\cA), (l^{\ast}\circ U_\loc^{\bfA^{\!1}})(\cB)) \nonumber\\
& \simeq & \Hom_\Sp(U_\loc(\cA), (\overline{U_\loc^h}\circ U_\loc^{\bfA^{\!1}})(\cB))\label{eq:1} \\
& \simeq & \Hom_\Sp(U_\loc(\cA), \mathrm{hocolim}_n\,U_\loc(\cB\otimes \Delta_n))\nonumber \\
& \simeq &  \mathrm{hocolim}_n\,\Hom_\Sp(U_\loc(\cA),U_\loc(\cB\otimes \Delta_n)) \label{eq:2} \\
& \simeq &  \mathrm{hocolim}_n\,\bbK(\cA^\op \otimes^\bfL (\cB \otimes \Delta_n)) \label{eq:3} \\
& = &  \bbK^h(\cA^\op \otimes^\bfL \cB)=: KH(\cA^\op \otimes^\bfL \cB)\,. \nonumber
\end{eqnarray}
Some explanations are in order: \eqref{eq:1} follows from isomorphism $l^\ast\simeq\overline{U_\loc^h}$ of Lemma~\ref{lem:auxiliar} below; \eqref{eq:2} follows from the compactness of the noncommutative motive $U_\loc(\cA)$ (see \cite[Corollary 8.7]{CT1}); and \eqref{eq:3} follows from 
 the weak equivalence
$$ \Hom_\Sp(U_\loc(\cA),U_\loc(\cB \otimes \Delta_n)) \simeq \bbK \rep_\dg(\cA, \cB \otimes \Delta_n)$$
(see \cite[Theorem~9.2]{CT1}) and from the existence of a Morita equivalence between $\rep_\dg(\cA,\cB \otimes \Delta_n)$ and $\cA^\op \otimes^\bfL (\cB \otimes \Delta_n)$ (see \cite[Lemma~5.9]{CT1}).
\begin{lemma}\label{lem:auxiliar}
The morphisms of derivators
\begin{eqnarray}\label{eq:2-morphisms}
l^\ast: \Mloc^{\bfA^{\!1}} \too \Mloc  && \overline{U_\loc^h}: \Mloc^{\bfA^{\!1}} \too \Mloc
\end{eqnarray}
are canonically isomorphic.
\end{lemma}
\begin{proof}
Consider the endomorphism $\mathsf{L}:= \overline{U_\loc^h} \circ l_!$ of $\Mloc$. Thanks to equivalence~\eqref{eq:equiv-Uloc}, the $2$-morphism $\eta: U_\loc \Rightarrow U_\loc^h$ extends to a $2$-morphism $\overline{\eta}: \Id \Rightarrow \mathsf{L}$. Consider the noncommutative motive $\mathsf{L}^{\bfA^{\!1}}:=\mathrm{hocolim}_n\, U_\loc(\Delta_n) \in \Mloc$. We claim that $\mathsf{L}(-)\simeq -\otimes^\bfL \mathsf{L}^{\bfA^{\!1}}$. Since these two endomorphisms preserve homotopy colimits and $\Mloc$ is generated by the noncommutative motives of the form $U_\loc(\cA)$, with $\cA$ a dg category, it suffices to show that $\mathsf{L}(U_\loc(\cA))\simeq U_\loc(\cA) \otimes^\bfL \mathsf{L}^{\bfA^{\!1}}$. This follows from the isomorphisms
\begin{eqnarray*}
\mathsf{L}(U_\loc(\cA)) & \simeq & U^h_\loc(\cA)  :=\mathrm{hocolim}_n U_\loc(\cA\otimes \Delta_n) \\
& \simeq & \mathrm{hocolim}_n (U_\loc(\cA) \otimes^\bfL U_\loc(\Delta_n))  \\
& \simeq & U_\loc(\cA) \otimes^\bfL \mathrm{hocolim}_n U_\loc(\Delta_n) =  U_\loc(\cA) \otimes^\bfL \mathsf{L}^{{\bf A\!}^1}\,.
\end{eqnarray*}
Under this identification, the evaluation of the $2$-morphism $\overline{\eta}$ at the noncommutative motive $U_\loc(\cA)$ corresponds to the following composition
$$ U_\loc(\cA) \stackrel{r}{\too} U_\loc(\cA) \otimes^{\bfL} U_\loc(k) \stackrel{\id\otimes \iota}{\too} U_\loc(\cA) \otimes^{\bfL} \mathsf{L}^{\bfA^{\!1}}\,,$$
where $r$ is the right isomorphism constraint and $\iota$ the canonical map. Let us now prove that the couple $(\mathsf{L}, \overline{\eta})$ defines a left Bousfield localization of $\Mloc$, \ie that the natural transformations $\mathsf{L}\overline{\eta}$ and $\overline{\eta}_\mathsf{L}$ are not only equal but moreover isomorphisms. Once again, since $\Mloc$ is generated by the noncommutative motives of the form $U_\loc(\cA)$, with $\cA$ a dg category, it suffices to show that the morphisms 
\begin{eqnarray*}
&U_\loc(\cA) \otimes^\bfL \mathsf{L}^{\bfA^{\!1}} \stackrel{r\otimes \id}{\too} U_\loc(\cA) \otimes^\bfL U_\loc(k)\otimes^\bfL \mathsf{L}^{\bfA^{\!1}} \stackrel{\id \otimes \iota \otimes \id}{\too} U_\loc(\cA) \otimes^\bfL \mathsf{L}^{\bfA^{\!1}} \otimes^\bfL \mathsf{L}^{\bfA^{\!1}}&\\
&U_\loc(\cA) \otimes^\bfL \mathsf{L}^{\bfA^{\!1}} \stackrel{\id\otimes r}{\too} U_\loc(\cA) \otimes^\bfL \mathsf{L}^{\bfA^{\!1}} \otimes^\bfL U_\loc(k) \stackrel{\id \otimes \id \otimes \iota}{\too} U_\loc(\cA) \otimes^\bfL \mathsf{L}^{\bfA^{\!1}} \otimes^\bfL \mathsf{L}^{\bfA^{\!1}} &
\end{eqnarray*}
are not only equal but moreover isomorphisms. The later claim follows from the isomorphisms $ \iota \otimes \id$ and $\id \otimes \iota$, which in turn follows from the cofinality of the maps $\Delta \stackrel{\id \times 0}{\too} \Delta \times  \Delta$ and $\Delta \stackrel{0 \times \id}{\too} \Delta \times \Delta$. On the other hand, the former claim follows from the commutativity of the following diagram
$$
\xymatrix{
U_\loc(\cA) \otimes^\bfL \mathsf{L}^{\bfA^{\!1}} \ar[r]^-{r\otimes \id} \ar@{=}[d]  & U_\loc(\cA) \otimes^\bfL U_\loc(k)\otimes^\bfL \mathsf{L}^{\bfA^{\!1}} \ar[r]^-{\id \otimes \iota \otimes \id} & U_\loc(\cA) \otimes^\bfL \mathsf{L}^{\bfA^{\!1}} \otimes^\bfL \mathsf{L}^{\bfA^{\!1}} \ar@{=}[d] \\
U_\loc(\cA) \otimes^\bfL \mathsf{L}^{\bfA^{\!1}} \ar[r]_-{\id\otimes r} & U_\loc(\cA) \otimes^\bfL \mathsf{L}^{\bfA^{\!1}} \otimes^\bfL U_\loc(k) \ar[r]_-{\id \otimes \id \otimes \iota}  \ar[u]^{\id \otimes \tau}_\sim & U_\loc(\cA) \otimes^\bfL \mathsf{L}^{\bfA^{\!1}} \otimes^\bfL \mathsf{L}^{\bfA^{\!1}}\,,
}
$$
where $\tau$ is the symmetry isomorphism constraint. Now, in order to prove that the morphisms \eqref{eq:2-morphisms} are isomorphic, it suffices by the general formalism of left Bousfield localization to show the following: a morphism in $\Mloc$ becomes an isomorphism after application of $\mathsf{L}$ if and only if it becomes an isomorphism after application of $l_!$. For this purpose it is enough to consider the morphisms $\overline{\eta}$. Once again, since $\mathsf{L}$ and $l_!$ are symmetric monoidal and homotopy colimit preserving, and $\Mloc$ is generated by the noncommutative motives of the form $U_\loc(\cA)$, with $\cA$ a dg category, we can restrict ourselves to the morphism $l_!(U_\loc(k) \to \mathrm{hocolim}_n\, U_\loc(\Delta_n))$. This is clearly an isomorphism since $U_\loc^{\bfA^{\!1}}=l_! \circ U_\loc$ is $\bfA^{\!1}$-homotopy invariant.
\end{proof}

Let us now prove the weak equivalence \eqref{eq:star-4}. Consider the distinguished triangle
$$ \Omega U^{{\bf A\!}^1}_\loc(\cA) \too l\backslash U^{{\bf A\!}^1}_\loc(\cA)  \too U^{{\bf A\!}^1}_\loc(\cA) \stackrel{\cdot l}{\too} U^{{\bf A\!}^1}_\loc(\cA)\,.$$
By applying to it the contravariant functor $\Hom_\Sp(-,U^{{\bf A\!}^1}_\loc(\cB))$ and using the weak equivalence \eqref{eq:star-3}, we obtain the following distinguished triangle of spectra
\begin{equation*}
KH(\cA^\op \otimes^\bfL \cB) \stackrel{\cdot l}{\to} KH(\cA^\op \otimes^\bfL \cB) \to \Hom_\Sp(l \backslash U^{{\bf A\!}^1}_\loc(\cA), U^{{\bf A\!}^1}_\loc(\cB)) \to \Sigma KH(\cA^\op \otimes^\bfL \cB)\,.
\end{equation*}
This triangle implies that $\Hom_\Sp(l \backslash U^{{\bf A\!}^1}_\loc(\cA), U^{{\bf A\!}^1}_\loc(\cB))$ is  the mod-$l$ Moore object of $KH(\cA^\op \otimes^\bfL \cB)$. Now, recall from \S\ref{sub:KH} that $KH(\cA^\op \otimes^\bfL \cB; \bbZ/l)$ is defined as $K(\cA^\op \otimes^\bfL \cB) \wedge^\bfL \bbS/l$. Using the distinguished triangle $\bbS \stackrel{\cdot l}{\to} \bbS \to \bbS/l \to \Sigma \bbS$, we conclude that $KH(\cA^\op \otimes^\bfL \cB;\bbZ/l)$ is also the mod-$l$ Moore object of $KH(\cA^\op \otimes^\bfL \cB)$. This achieves the proof of Theorem \ref{thm:main2}.
\section{Proof of Corollary \ref{cor:main}}
Recall from \S\ref{sub:smooth} that since by assumption $X$ is a smooth proper $k$-scheme, the dg category $\perf_\dg(X)$ is smooth and proper. Hence, Theorem~\ref{thm:main2} (with $\cA=\perf_\dg(X)$ and $\cB=\perf_\dg(Y)$) gives rise to the weak equivalence
$$\Hom_\Sp(U_\loc^{{\bf A\!}^1}(\perf_\dg(X)), U_\loc^{{\bf A\!}^1}(\perf_\dg(Y))) \simeq  KH(\perf_\dg(X)^\op \otimes^\bfL \perf_\dg(Y))\,. $$
Thanks to \cite[Proposition\footnote{In {\em loc. cit.} we assumed $X$ and $Y$ to be separated. However, the same result holds with $X$ and $Y$ quasi-separated.}~8.2]{Regularity} (with $E=KH$) and the Morita equivalence $\perf_\dg(X)^\op \simeq \perf_\dg(X)$, one concludes that the right-hand-side identifies with $KH(\perf_\dg(X \times Y))$. The proof follows now from Proposition \ref{prop:agreementKH}(ii).
\section{Proof of Theorem \ref{thm:main22}}
Let $\overline{KV}, \overline{E}: \Mot_\add^{{\bf A\!}^1} \to \HO(\Sp)$ be the homotopy colimit preserving morphisms of derivators associated to $KV,E$ under equivalence \eqref{eq:equiv-UAdd-main}. Note that $\Nat_\Sp(KV,E) \simeq \Nat_\Sp(\overline{KV},\overline{E})$. Now, consider the following sequence of weak equivalences
$$ \Nat_\Sp(\overline{KV},\overline{E}) \simeq \Nat_\Sp(\Hom_\Sp(U_\add^{{\bf A\!}^1}(k),-), \overline{E}) \simeq \overline{E}(k) \simeq E(k)\,.$$
The first one follows from Theorem~\ref{thm:main2} (with $\cA=k$), the second one follows from the $\Sp$-enriched Yoneda lemma, and the third one follows from $\overline{E} \circ U_\loc^{{\bf A\!}^1} \simeq E$. This implies the left-hand-side of \eqref{eq:nat-1}. The right-hand-side is obtained by applying the functor $\pi_0(-)$. Finally, the proof of the localizing case is similar
\section{Proof of Theorem~\ref{thm:main3}}\label{sec:proof-main3}
Let $ch(A): K(A) \to HP(A)$ be the classical Chern character from the algebraic $K$-theory of $A$ to the periodic cyclic homology of $A$. Consider the induced map
\begin{equation}\label{eq:induced-map}
\mathrm{hocolim}_n (K(\Delta_nA) \stackrel{ch(\Delta_n A)}{\too} HP(\Delta_n A))\,,
\end{equation}
where $\Delta_nA:= A[t_0, \ldots, t_n]/(\sum_{i=0}^n t_i -1)A$. As explained in the proof of Proposition \ref{prop:agreementKH}(i), the left-hand-side of \eqref{eq:induced-map} identifies with $KV(A)$. On the other hand, since $HP$ is ${\bf A\!}^1$-homotopy invariant, the right-hand-side identifies with $HP(A)$. Weibel's homotopy Chern characters $KV_n(A) \to HP_n(A), n \geq 1$, are obtained from \eqref{eq:induced-map} by applying the (stable) homotopy group functors $\pi_n(-), n \geq 1$; see \cite[\S5]{Weibel-Nil}.

Now, consider the following commutative diagram
\begin{equation}\label{eq:diagram-big}
\xymatrix{
\HO(\dgcat) \ar[d]^-{U_\add} \ar@/_2pc/[dd]_-{U_\add^{{\bf A\!}^1}} \ar[rr]^-{HP^{\mathrm{flt}}} && \HO(\Sp) \\
\Madd \ar@/_0.7pc/[urr]_-{\overline{HP^{\mathrm{flt}}}} \ar[d]^-{l_!} && \\
\Madd^{{\bf A\!}^1} \ar@/_1pc/[uurr]_-{\overline{\overline{HP^{\mathrm{flt}}}}} && \,,
}
\end{equation}
where $\overline{HP^{\mathrm{flt}}}$ and $\overline{\overline{HP^{\mathrm{flt}}}}$ are the homotopy colimit preserving morphism of derivators induced from (the additive version of) \eqref{eq:equiv-Uloc} and \eqref{eq:equiv-UAdd-main}, respectively. Note that the composition $ch^{{\bf A\!}^1}(\cA):KV(\cA) \to HP^{\mathrm{flt}}(\cA) \stackrel{\epsilon}{\to} HP(\cA)$ identifies with
$$
\Hom_\Sp(U_\add^{{\bf A\!}^1}(k), U_\add^{{\bf A\!}^1}(\cA)) \to \Hom_\Sp(HP(k),HP^{\mathrm{flt}}(\cA)) \to \Hom_\Sp(HP(k),HP(\cA))\,,
$$
where the left-hand-side map is induced by $\overline{\overline{HP^{\mathrm{flt}}}}$ and the right-hand-side one by the counit $2$-morphism $\epsilon$. Since $\Madd^{{\bf A\!}^1}$ is a left Bousfield localization of $\Madd$, we have by adjunction and compactness of $U_\add(k)$ the following weak equivalences
\begin{eqnarray*}
\Hom_\Sp(U_\add^{{\bf A\!}^1}(k), U_\add^{{\bf A\!}^1}(\cA)) & \simeq& \Hom_\Sp(U_\add(k), \mathrm{hocolim}_n U_\add (\cA\otimes \Delta_n))\\
& \simeq & \mathrm{hocolim}_n \Hom_\Sp(U_\add(k),U_\add(\cA \otimes \Delta_n))\,.
\end{eqnarray*}
On the other hand, since $HP^{\mathrm{flt}}$ and $HP$ are ${\bf A\!}^1$-homotopy invariant, we have
\begin{eqnarray*}
\Hom_\Sp(HP(k),HP^{\mathrm{flt}}(\cA)) & \simeq & \mathrm{hocolim}_n\Hom_\Sp(HP(k), HP^{\mathrm{flt}}(\cA \otimes \Delta_n)) \\
\Hom_\Sp(HP(k),HP(\cA)) & \simeq & \mathrm{hocolim}_n \Hom_\Sp(HP(k),  HP(\cA \otimes \Delta_n)) \,.
\end{eqnarray*}
As a consequence, $ch^{{\bf A\!}^1}(\cA)$ identifies with
\begin{equation}\label{eq:induced-map-1}
\mathrm{hocolim}_n (\Hom_\Sp(U_\add(k),U_\add(\cA \otimes \Delta_n)) \to \Hom_\Sp(HP(k), HP(\cA\otimes \Delta_n)))\,,
\end{equation}
where the maps are now induced by $\overline{HP^{\mathrm{flt}}}$ and $\epsilon$. Let us now prove that \eqref{eq:induced-map-1}=\eqref{eq:induced-map} when $\cA=A$. This clearly achieves the proof. In order to do so, consider the following commutative diagram
$$
\xymatrix{
\HO(\dgcat) \ar[d]_-{U_\add} \ar[rr]^-{P \circ M} && \HO(k[u]\text{-}\mathrm{Comod}) \ar[rr]^-{\Hom_\Sp(k[u],-)} && \HO(\Sp) \\
\Madd \ar[urr]_-{\overline{P \circ M}} & && & \,,
}
$$
where $\overline{P \circ M}$ is the homotopy colimit preserving morphism of derivators induced from (the additive version of) \eqref{eq:equiv-Uloc}. Recall from \S\ref{sec:HP} that the upper horizontal composition is $HP$. Given a dg category $\cA$, consider the composition of the map
\begin{equation}\label{eq:comp1}
\Hom_\Sp(U_\add(k),U_\add(\cA \otimes \Delta_n)) \too \Hom_\Sp(k[u],(P\circ M)(\cA \otimes \Delta_n))
\end{equation}
induced by $\overline{P \circ M}$ with the map
\begin{equation}\label{eq:comp2}
\Hom_\Sp(k[u],(P\circ M)(\cA \otimes \Delta_n)) \too \Hom_\Sp(HP(k),HP(\cA\otimes \Delta_n))
\end{equation}
induced by $\Hom_\Sp(k[u],-)$. As proved in \cite[Theorem~2.8]{Products} \cite[\S5]{Criterion}, the composition $\eqref{eq:comp2} \circ \eqref{eq:comp1}$ agrees with the Chern character $ch(\Delta_nA) : K(\Delta_nA) \to HP(\Delta_nA)$ when $\cA=A$. Hence, in order to prove the equality \eqref{eq:induced-map-1}=\eqref{eq:induced-map}, it suffices to show that the following diagram is commutative (up to weak equivalence)
\begin{equation}\label{eq:square-key}
\xymatrix{
\Hom_\Sp(U_\add(k),U_\add(\cA \otimes \Delta_n)) \ar[d] \ar[r]^-{\eqref{eq:comp1}} & \Hom_\Sp(k[u],(P \circ M) (\cA \otimes \Delta_n)) \ar[d]^-{\eqref{eq:comp2}} \\
\Hom_\Sp(HP(k),HP^{\mathrm{flt}}(\cA \otimes \Delta_n)) \ar[r] & \Hom_\Sp(HP(k),HP(\cB \otimes \Delta_n))\,,
}
\end{equation} 
where the left vertical map is induced by $\overline{HP^f}$ and the bottom horizontal map by $\epsilon$. Let us assume first that $\cA$ is homotopically finitely presented. Since the $k$-algebra $\Delta_n$ (considered as a dg category) is clearly homotopically finitely presented, $\cA \otimes \Delta_n$ is also homotopically finitely presented; see \cite[Theorem~4.4]{CT1}. Hence, thanks to Proposition~\ref{prop:general}(i), the bottom horizontal map is an isomorphism. We now claim that, via the adjunction \eqref{eq:desired-adj} below, we have a $2$-isomorphism
$$ \Psi(\Hom_\Sp(k[u],-) \circ \overline{P \circ M}) \simeq \overline{HP^{\mathrm{flt}}}\,.$$
Thanks to equivalence \eqref{eq:equiv-final-2} and adjunction \eqref{eq:desired-adj1}, this follows from the fact that $\Hom_\Sp(k[u],-) \circ \overline{P \circ M}$ and $\overline{HP^{\mathrm{flt}}}$ agree with $HP$ when precomposed with $h: \dgcat_f[S^{-1}] \to \Madd$ and from the fact that $\overline{HP^{\mathrm{flt}}}$ is homotopy colimit preserving. Making use of Proposition~\ref{prop:aux}, we then conclude that \eqref{eq:square-key} is commutative. Let us now assume that $\cA$ is an arbitrary dg category. As proved in \cite[Proposition~3.6(iii)]{CT}, there exists a filtered direct system of finite dg cells $\{\cB_j\}_{j\in J}$ such that $\mathrm{hocolim}_j \cB_j \simeq \cA$. Consequently, we have the weak equivalences
\begin{eqnarray*}
\Hom_\Sp(U_\add(k),U_\add(\cA \otimes \Delta_n)) &\simeq & \Hom_\Sp(U_\add(k),U_\add(\mathrm{hocolim}_j \cB_j \otimes \Delta_n))  \\
& \simeq & \Hom_\Sp(U_\add(k),\mathrm{hocolim}_j U_\add( \cB_j \otimes \Delta_n)) \\
& \simeq & \mathrm{hocolim}_j \Hom_\Sp(U_\add(k), U_\add( \cB_j \otimes \Delta_n))\,.
\end{eqnarray*}
Therefore, in order to prove that \eqref{eq:square-key} is commutative, it suffices to show that its precomposition with the maps 
\begin{eqnarray*}
\Hom_\Sp(U_\add(k),U_\add(\cB_j\otimes \Delta_n)) \too \Hom_\Sp(U_\add(k),U_\add(\cA \otimes \Delta_n)), && j \in J
\end{eqnarray*}
is commutative. This follows automatically from the functoriality of diagram \eqref{eq:square-key} on $\cA$ and from the previous case.  
\begin{proposition}\label{prop:aux}
Given any triangulated derivator $\bbD$, one has an adjunction
\begin{equation}\label{eq:desired-adj}
\xymatrix{
\uHom(\Madd,\bbD) \ar@<1ex>[d]^-{\Psi} \\
*+<3ex>{\uHom_!(\Madd,\bbD)} \ar@{^{(}->}@<1ex>[u]
}
\end{equation}
Given $E' \in \uHom(\Madd,\bbD)$, the evaluation of the counit $2$-morphism $\Psi(E') \Rightarrow E'$ at every homotopically finitely presented dg category is an isomorphism.
\end{proposition}
\begin{proof}
Recall first from (the additive version of) Remark~\ref{rk:Qmodel} that $\Madd$ admits a Quillen model $\Madd^Q$ given by $L_\add\!\Fun(\dgcat_f^\op,\Sp)$, where $\add$ is a set of morphisms implementing the additive property. When $\bbD$ is a triangulated derivator, the equivalence \eqref{eq:equivalence-key} (with $\sSet$ replaced by $\Sp$)
\begin{equation}\label{eq:equiv-final-1}
h^\ast: \uHom_!(\HO(L_S\!\Fun(\dgcat_f^\op, \Sp)),\bbD) \stackrel{\sim}{\too} \uHom(\dgcat_f[S^{-1}],\bbD) 
\end{equation}
holds also; see \cite[Theorem~3.1 and \S8]{Additive}. By further localizing $L_S\!\Fun(\dgcat_f^\op,\Sp)$ with respect to $\add$, we obtain the Quillen model $\Madd^Q$. Since every split short exact sequence of dg categories is Morita equivalent to a filtered homotopy colimit of split short exact sequences whose components are finite dg cells (see \cite[Proposition~13.2]{Additive}), \eqref{eq:equiv-final-1} give then rise to the following equivalence
\begin{equation}\label{eq:equiv-final-2}
h^\ast: \uHom_!(\Madd,\bbD) \stackrel{\sim}{\too} \uHom_{\mathsf{sses}}(\dgcat_f[S^{-1}],\bbD) \,,
\end{equation}
where the right-hand-side denotes the category of morphisms of derivators that send {\bf s}plit {\bf s}hort {\bf e}xact {\bf s}equences of dg categories to direct sums. As in \eqref{eq:adj-2}, we obtain then the following adjunction 
\begin{equation}\label{eq:desired-adj1}
\xymatrix{
\uHom(\Madd, \bbD) \ar@<1ex>[d]^-{\Psi}  & E' \ar@{|->}[d]\\
*+<3ex>{\uHom_!(\Madd, \bbD)} \ar@{^{(}->}@<1ex>[u] & \Psi(E'):= \overline{E' \circ h}\,,
}
\end{equation}
where $\overline{E' \circ h}$ is the unique homotopy colimit preserving morphism of derivators corresponding to $E' \circ h$ under the above equivalence \eqref{eq:equiv-final-2}. This establishes the desired adjunction \eqref{eq:desired-adj}. The second claim is now clear from the construction of the right adjoint $\Psi$ and from the fact that every homotopically finitely presented dg category is a retract (in the homotopy category $\Ho(\dgcat)$) of a finite dg cell.
\end{proof}

\appendix

\section{Grothendieck derivators}\label{sub:derivators}
The theory of derivators allow us to state and prove precise universal properties. The original reference is Grothendieck's manuscript \cite{Grothendieck}; consult the Appendices of \cite{CT1,CT} for shorter and more didactic accounts. Roughly speaking, a derivator $\bbD$ consists of a
strict contravariant $2$-functor from the $2$-category $\Cat$ of small
categories to the $2$-category $\CAT$ of all categories
\begin{eqnarray*}
\bbD: \Cat^\op \longrightarrow \CAT && I \mapsto \bbD(I)
\end{eqnarray*}
subject to several natural axioms. The
essential example to keep in mind is the derivator $\bbD=\HO(\cM)$
associated to a Quillen model category~$\cM$
and defined for every small category~$I$ by $\HO(\cM)(I):=\Ho(\Fun(I^\op,\cM))$.
Let $e$ be the $1$-point category with only one object and one
identity morphism. By definition, $\bbD(e)$ is called the {\em base category} of the derivator $\bbD$. Heuristically, it is the
basic ``derived" category under consideration. For instance, if $\bbD=\HO(\cM)$ then
$\bbD(e)=\Ho(\cM)$. 

A derivator $\bbD$ is called {\em triangulated} if $\bbD(I)$ is a triangulated category for every small category $I$. For example, the derivator $\HO(\cM)$ associated to a stable Quillen model category $\cM$ is triangulated. As explained in \cite[\S A.1]{CT}, every triangulated derivator $\bbD$ is naturally enriched $\Hom_\Sp(-,-)$ over spectra. Given derivators $\bbD,\bbD'$, we will write $\uHom(\bbD,\bbD')$ for the category of morphisms of derivators, $\uHom_{\mathrm{flt}}(\bbD,\bbD')$ for the full subcategory of filtered homotopy colimit preserving morphisms of derivators, and $\uHom_!(\bbD,\bbD')$ for the full subcategory of homotopy colimit preserving morphisms of derivators. Finally, given morphisms of derivators $E,E': \bbD \to \bbD'$, with $\bbD'$ triangulated, we will write $\Nat_\Sp(E,E')$ for the spectrum of natural transformations (\ie $2$-morphisms) and $\Nat(E,E')$ for $\pi_0 \Nat_\Sp(E,E')$.

\end{document}